\documentclass{amsart}
\usepackage{amssymb}
\usepackage{amsmath}

\usepackage{color}

\begin{document}
%\date{\version}
\newtheorem{theorem}{Theorem}[section]
\newtheorem{lemma}[theorem]{Lemma}
\newtheorem{remark}[theorem]{Remark}
\newtheorem{definition}[theorem]{Definition}
\newtheorem{corollary}[theorem]{Corollary}
\newtheorem{example}[theorem]{Example}
\def\qedbox{\hbox{$\rlap{$\sqcap$}\sqcup$}}
\makeatletter
  \renewcommand{\theequation}{%
   \thesection.\alph{equation}}
  \@addtoreset{equation}{section}
 \makeatother
\title[Strong curvature homogeneity of type (1,3)]
{Singer invariants and strongly curvature homogeneous manifolds of type (1,3)}
\author{Corey Dunn*, Cullen McDonald}

\begin{address}{CD: Mathematics Department, California State University at San Bernardino,
San Bernardino, CA 92407, USA. Email: \it
cmdunn@csusb.edu.}\end{address}

\begin{address}{CM: Mathematics Department, Beloit College,
Beloit, WI, USA. Email: \it
mcdonaldc@beloit.edu.}\end{address}

\begin{abstract}
We extend the definition of  curvature homogeneity of type $(1,3)$ to include the possibility that there is a homothety between any two points of a manifold preserving the first $r$ covariant derivatives of the curvature operator simultaneously; we call this strong curvature homogeneity of type (1,3) up to order $r$.  We characterize these properties in terms of model spaces.  In addition, we also present two families of three-dimensional Lorentzian metrics on Euclidean space  to exhibit the behavior of this property.  The first example is curvature homogeneous of type (1,3) of all orders, but is not locally homogeneous.  Within this first family, being strongly curvature homogeneous of type (1,3) up to order $1$ implies local homogeneity.  The second example is strongly curvature homogeneous of type (1,3) up to order one, and is not locally homogeneous, showing that this new definition is not a trivial one.  Within this family, being strongly curvature homogeneous of type (1,3) up to order 2 implies local homogeneity.  
\end{abstract}

%%%%%%%%%%%%%%%%%

\keywords{curvature homogeneous, strongly curvature homogeneous of type $(1,3)$, Singer invariant. \newline
* corresponding author} \maketitle

\section{Introduction}         %%%%%%%%%%%%%%%%%%%%

Let $(M,g)$ be a smooth pseudo-Riemannian manifold of dimension $n$, and let $T_PM, T^*_PM$ and $g_P$ denote the tangent space of $M$ at $P \in M$, the cotangent space at $P$, and the metric $g$ restricted to the point $P$, respectively.  If $\nabla$ is the Levi-Civita connection of $(M,g)$, then one creates the Riemann curvature operator $\mathcal{R}_P \in T_PM \otimes (T_P^*M)^{\otimes3}$ and the Riemann curvature tensor $R_P \in \otimes^4 T_P^*M$ at $P$ for  $X, Y, Z, W \in T_PM$ as
\begin{equation*}\begin{array}{r c l}
\mathcal{R_P}(X, Y)Z &=& \nabla_X\nabla_YZ - \nabla_Y\nabla_XZ - \nabla_{[X,Y]}Z, {\rm\ and} \\
R_P(X, Y, Z, W) & = & g_P(\mathcal{R}_P(X, Y)Z, W).
\end{array}
\end{equation*}
In addition, let $\nabla^k\mathcal{R}_P$ and $\nabla^kR_P$ be the $k^{\rm th}$ covariant derivative of the curvature operator and curvature tensor at $P$, respectively.  

In this introductory section, we give a brief background of the study of curvature homogeneity, introduce model spaces as a preliminary notion which we will need, and finally, give an outline of the paper which summarizes our main results.

\subsection{Background}

Singer began the study of curvature homogeneity in 1960 \cite{S1960}, and since then there have been several generalizations studied by a number of authors (see \cite{G2007}, and \cite{K2011, K2013} for recent work relevant to this article).  We give the definitions of \emph{curvature homogeneity} \cite{G2007} and \emph{curvature homogeneity of type $(1,3)$} \cite{K2011, K2013} in Definition \ref{basicdefinitions}:

\begin{definition}  \label{basicdefinitions}
Let $(M,g)$ be a smooth pseudo-Riemannian manifold.
\begin{enumerate}
\item  $(M,g)$ is \emph{$r-$curvature homogeneous} if, for every $P, Q \in M$, there is a linear isometry $F:T_PM \to T_QM$ so that $F^*\nabla^kR_Q = \nabla^k R_P$ for every $k = 0, 1, \ldots, r$.  If $(M,g)$ is $r-$curvature homogeneous, we say $(M, g)$ is $CH_r$, and if it is additionally not $(r+1)-$curvature homogeneous, then we say $(M,g)$ is \emph{properly $CH_r$}.
\item  $(M,g)$ is \emph{$r-$curvature homogeneous of type $(1,3)$} if, for every $k = 0, 1, \ldots, r$,   there is a homothety $h_k:T_PM \to T_QM$ for every $P, Q \in M$ so that $h_k^*\nabla^k\mathcal{R}_Q = \nabla^k\mathcal{R}_P$.  If $(M,g)$ is $r-$curvature homogeneous of type $(1,3)$ then we say $(M,g)$ is $CH_r(1,3)$, and if it is additionally not $(r+1)-$curvature homogeneous of type $(1,3)$, then we say $(M,g)$ is \emph{properly $CH_r(1,3)$}.
\end{enumerate}
\end{definition}

Several remarks are in order concerning these definitions.  First, the curvature homogeneity condition of type $(1,3)$ appeared recently (see \cite{K2011, K2013}) and focuses on the preservation of the covariant derivatives of the curvature operator (of tensor type $(1,3)$) by a homothety (an isometry followed by a dilation), as opposed to the curvature tensor (of type $(0,4)$), which justifies the ``$(1,3)$'' in its definition.  This is not to be confused with the property of \emph{affine curvature homogeneity} since one has only a connection and no metric in that situation (see \cite{O1996, O1997} for work in this area).  Second, it is important to note--and it is the motivation of this article--that in the definition of $r-$curvature homogeneous of type $(1,3)$, the homotheties $h_k$ may have no relationship to each other, and we define \emph{strongly $r-$curvature homogeneous of type (1,3)} to distinguish this special phenomenon:

\begin{definition}  \label{stronglych}
Let $(M,g)$ be a smooth pseudo-Riemannian manifold.  $(M,g)$ is \emph{strongly $r-$curvature homogeneous of type $(1,3)$} if for every $P, Q \in M$,  there exists a homothety $h:T_PM \to T_QM$ so that  $h^*\nabla^k\mathcal{R}_Q = \nabla^k\mathcal{R}_P$ for all $k = 0, \ldots, r$.  If $(M,g)$ is strongly $r-$curvature homogeneous of type $(1,3)$ then we say $(M,g)$ is $SCH_r(1,3)$, and if it is additionally not strongly $(r+1)-$curvature homogeneous of type $(1,3)$, then we say $(M,g)$ is \emph{properly $SCH_r(1,3)$}.
\end{definition}

As it has been remarked by previous authors, the property of being properly $SCH_r(1,3)$ seems very restrictive (see the discussion on page 130 of \cite{K2013}).  The second example presented in this paper is one that is properly $SCH_1(1,3)$  (see Theorem \ref{Mh}); this phenomenon has, as of yet, not yet been known to have been possible, and demonstrates that the property $SCH_r(1,3)$ is an interesting one.  

There are obvious relationships between some of the types of curvature homogeneity listed in Definitions \ref{basicdefinitions} and \ref{stronglych}.  First, if a manifold is locally homogeneous, then it satisfies both of (1) and (2) for all $r$ in Definition \ref{basicdefinitions}, and is $SCH_r(1,3)$ for all $r$.   Second, if a manifold is $CH_r$ then it is  $CH_r(1,3)$, and $SCH_r(1,3)$.  Third, if a manifold is $SCH_r(1,3)$ then it is also $CH_r(1,3)$.

A foundational result in the broad study of curvature homogeneity is credited to Singer \cite{S1960} (in the Riemannian case) and by Podesta and Spiro  \cite{PS04} (in the pseudo-Riemannian case):

\begin{theorem} \label{Singer}
Let $(M,g)$ be a smooth pseudo-Riemannian manifold of dimension $n$ and signature $(p,q)$.  There is a constant $k_{p, q} \leq n(n-1)/2$ so that if $(M,g)$ is $k_{p,q}-$curvature homogeneous, then it is locally homogeneous.
\end{theorem}

The number $k_{p,q}$ is known as the \emph{Singer number} of a manifold of signature $(p,q)$.  Opozda has a similar result for affine curvature homogeneity \cite{O1997}.  Generally speaking, one wants to know what $k_{p,q}$ is for any $p$ and $q$.  A \emph{Singer-type} number $k$ is a number where if a manifold satisfies a certain type of curvature homogeneity up to order $k$, then it is locally homogeneous (or locally affine homogeneous, in the case of affine curvature homogeneity).  It is clear that $k_{0,2} = k_{1,1} = 0$.  It is known \cite{SSV1, SSV2} that $k_{0, 3} = k_{0,4} = 1$, and that \cite{BD} $k_{1,2} = 2$ (this number shall play a role in Assertion (4) of Theorem \ref{Mh}).  Recent work \cite{MP2009} shows that $k_{1,3} = 3$.   More generally, the examples in \cite{GS2004} show that $k_{p,q} \geq \min\{p, q\}$ and that \cite{Gr1986} $k_{0,n} \leq \frac32n - 1$.  There are still no known examples of properly $CH_1$ Riemannian manifolds, and it is conjectured that are are none.

\subsection{Model spaces}

We phrase many of our results in terms of model spaces. This short subsection provides a definition of model spaces and relates them to curvature homogeneity.

Let $V$ be a finite dimensional real vector space, and let $V^* = {\rm Hom}_{\mathbb{R}}(V, \mathbb{R})$ be its dual. If $\alpha_i \in \otimes^{n_i} V^*$ for $i = 1, \ldots, m$ and nonnegative integers $n_i$, then the tuple $(V, \alpha_1, \ldots, \alpha_m)$ is called a \emph{model space}.  For example, if $(M,g)$ is a smooth pseudo-Riemannian manifold as defined earlier, then we shall denote
$$
\mathcal{M}_r(P) := (T_PM, g_P, R_P, \ldots, \nabla^mR_P),
$$
and point out that $\mathcal{M}_r(P)$ is a model space.  Here, $\alpha_1 = g_P \in \otimes^2 T_P^*M$, and $\alpha_{k} = \nabla^{k-2}R_P \in \otimes^{4+k}T_P^*M$ for $k \geq 2$.  These model spaces provide an algebraic portrait of the curvature of a manifold at any point, and motivates the terminology.  In this situation, the tensors $\alpha_i$ have implicit symmetries:  $\alpha_1$ is symmetric, $\alpha_2$ is an algebraic curvature tensor, and $\alpha_i$ for $i \geq 3$ are algebraic covariant derivative curvature tensors \cite{G2007}.

Two model spaces $(V, \alpha_1, \ldots, \alpha_m)$ and $(W, \beta_1, \ldots, \beta_m)$ are \emph{isomorphic} if there is a vector space isomorphism $F:V \to W$ so that $F^*\beta_i = \alpha_i$, where the standard notation $F^*$ denotes precomposition in each input.  

Returning now to the geometric setting,  we apply this terminology to curvature homogeneity.  Let $A_k \in \otimes^{4+k}V^*$ share the same universal symmetries as the $k^{{\rm th}}$ covariant derivative as the Riemann curvature tensor, and let $\varphi$ be a nondegenerate inner product on $V$.  Referring to Definition \ref{basicdefinitions}, it is easy to see that a manifold is $CH_r$ if for each $P$ there is an isomorphism
$$
\mathcal{M}_r(P) \cong (V, \varphi, A_0, \ldots, A_r).
$$

Kowalski and Van\u{z}urov\'{a} \cite{K2013} have reformulated the property of being $CH_r(1,3)$ in terms of the curvature tensors $\nabla^kR$ instead of the curvature operators $\nabla^k\mathcal{R}$.  For convenience, we follow their discussion and introduce a condition $Q(k)$ for a smooth manifold $(M,g)$ as follows:

\begin{tabular}{r c p{4in}}
$Q(k)$& : & for each fixed $P \in M$, there exists a smooth positive function $\psi_k:M \to \mathbb{R}$ with $\psi(P) = 1$ so that for every $Q\in M$ there is a linear isometry $F_k:T_PM \to T_QM$ so that $\nabla^k R_P = \psi_k(Q)F_k^*\nabla^kR_Q$.
\end{tabular}

\begin{theorem} \label{KV}
The smooth manifold $(M,g)$ is $CH_r(1,3)$ if and only if for each fixed point $P\in M$, the property $Q(k)$ is satisfied for each $k = 0, \ldots, r$.
\end{theorem}

In other words, they have proven that $(M,g)$ is $CH_r(1,3)$ if and only if for every $P \in M$ there are isomorphisms
$$
(T_PM, g_P, \nabla^kR_P) \cong (V, \varphi, \psi_k(P)A_k),
$$
where $\varphi$ is a nondegenerate inner product sharing the same signature as $g$,  $A_k \in \otimes^{4+k}V^*$ sharing the same universal symmetries as $\nabla^kR_P$, and the $\psi_k:M \to \mathbb{R}$ are smooth positive functions on $M$.

Since, in the definition of $CH_r(1,3)$, there is no requirement that the homotheties $h_k$ be related, one follows the proof of Theorem \ref{KV} in \cite{K2013} to see that there is no requirement that the isomorphisms $F_k$ have any relation.  It is one of the goals of this paper to understand how restrictive the geometry of a manifold must be in the event that the homotheties $h_k$ (and therefore the isometries $F_k$) are all equal, which motivates our definition of strong curvature homogeneity of type (1,3) up to order $r$.  We prove a result similar to Theorem \ref{KV} in Section 2 characterizing the $SCH_r(1,3)$ property in terms of model spaces:

\begin{theorem}  \label{SCHchar}
Let $(M,g)$ be a smooth manifold, and let $(V, \varphi, A_0, \ldots, A_r)$ be a model space, where $\varphi$ is a nondegenerate inner product, and $A_k \in \otimes^{4+k}V^*$ for all $k$.  The manifold $(M,g)$ is $SCH_r(1,3)$ if and only if there exists a smooth positive function $\psi:M \to \mathbb{R}$ and for each $P \in M$ there are isomorphisms
$$
\mathcal{M}_r(P)  \cong (V, \varphi, \psi(P)A_0, \psi^{3/2}(P)A_1, \ldots, \psi^{\frac{1}{2}(r+2)}(P)A_r).
$$
\end{theorem}

\subsection{Outline of the paper}

The following is a brief outline of the paper.  In Section \ref{section2}, we prove Theorem \ref{SCHchar}.  The remainder of the paper is devoted to the exposition of two examples which we define here.

Let $M = \mathbb{R}^3$, with coordinates $(t, x, y)$.  Let $h$ and $f$ be smooth real valued functions on $M$, where $f = f(x)$ is a function of only $x$ and $h = h(t)$ is only a function of $t$.  Let $g_f$ and $g_h$ be metrics on $M$ whose nonzero terms are given below on the coordinate frames:
$$
g_f(\partial_t, \partial_t) = e^{2f}, \quad g_f(\partial_x, \partial_y) = 1,
$$
and
$$
g_h(\partial_t, \partial_t) = g_h(\partial_x, \partial_y) = 1, \quad g_h(\partial_x, \partial_x) = -2h.
$$
In Section \ref{section3}, we prove the following concerning the manifolds $(M, g_f)$:

\begin{theorem}  \label{Mf}
Let $(M,g_f)$ be as above,  set $\Delta = f'' + (f')^2$, and $\Delta^{(k)} = \frac{d^k\Delta}{dx^k}$.   Suppose $\Delta$ and $\Delta^{(1)}$ are nonvanishing.
\begin{enumerate}
\item  The manifolds $(M, g_f)$ are $CH_0$.
\item  If $\Delta^{(k)} 0$ are nonvanishing for all $k$, then $(M, g_f)$ is $CH_r(1,3)$ for every $r$.
\item  The quantity $(\Delta^{(1)})^2$ is an isometry invariant, and if it is nonconstant and nonvanishing, then $(M, g_f)$ is not $CH_1$, and hence not locally homogeneous.
\item  If $f$ is analytic and $(M, g_f)$ is $SCH_1(1,3)$, then it is $CH_r$ for every $r$.  As a result,  $SCH_1(1,3)$ implies local homogeneity.
\end{enumerate}
\end{theorem}

We remark that Assertions (2) and (3) of Theorem \ref{Mf} shows that there is no finite Singer-type invariant relating to $CH_r(1,3)$ manifolds.

In Section \ref{section4}, we prove the following concerning the manifolds $(M, g_h)$:

\begin{theorem}  \label{Mh}
Let  $(M, g_h)$ be as above, and suppose that $h''$ is nonvanishing.
\begin{enumerate}
\item   $(M, g_h)$ is $CH_0$.
\item  If $h'''$ is nonvanishing, then $(M, g_h)$ is $SCH_1(1,3)$.
\item  The quantity $\left(\frac{h'''}{h''}\right)^2$ is an isometry invariant, and if it is nonconstant, then $(M, g_h)$ is not $CH_1$, and is therefore not locally homogeneous.
\item  Suppose $h'$ and $h'''$ are nonvanishing.  If $(M, g_h)$ is $SCH_2(1,3)$, then it is $CH_2$.  As a result,  if $h$ is analytic, then $SCH_2(1,3)$ implies local homogeneity.
\end{enumerate}
\end{theorem}

We remark that Assertions (2) and (4) of Theorem \ref{Mh} indicate that there may be a finite Singer-type invariant relating to $SCH_r(1,3)$ manifolds, and we conjecture that such a number exists.  We conclude the paper with a short section laying out ideas for further study.

\section{A characterization of strong curvature homogeneity of type (1,3) up to order $r$}  \label{section2}
This section is devoted to the proof of Theorem \ref{SCHchar}, and it follows very similarly to the proof of Proposition 0.1 in \cite{K2013}.  We recall for the purposes of this proof that if $h:T_PM \to T_QM$ is an invertible linear map, then the pullback of the curvature operator $\nabla^k\mathcal{R}_Q$ is defined as follows for $x, y, z, v_1, \ldots, v_k \in T_PM$:
$$
h^*\nabla^k\mathcal{R}_Q(x, y, z; v_1, \ldots, v_k)  = h^{-1}\nabla^k\mathcal{R}_Q(hx, hy, hz; hv_1, \ldots, hv_k).
$$
We have $h^*\nabla^k\mathcal{R}_Q = \nabla^k\mathcal{R}_P$ if and only if
$$
h(\nabla^k\mathcal{R}_P(x, y, z; v_1, \ldots, v_k)) = \nabla^k\mathcal{R}_Q(hx, hy, hz; hv_1, \ldots, hv_k).
$$

\begin{proof}[Proof of Theorem \ref{SCHchar}]   %%%%%%
Let $Q\in M$ be any point, and define $V = T_QM$, $\varphi = g_Q$, and $A_k = \nabla^kR_Q$ for $k = 0, \ldots, r$, so that $(V, \varphi, A_0, \ldots, A_r)$ is a model space.  For convenience, denote $\nabla^k\mathcal{R}_Q = \mathcal{A}_k$.   Suppose $(M, g)$ is $SCH_r(1,3)$, and $h:T_PM \to V$ a homothety with $h^*\mathcal{A}_k = \nabla^k\mathcal{R}_P$ for all $k = 0, \ldots, r$.  Suppose $h = \lambda F$, where $F:T_PM \to V$ is an isometry and in a helpful abuse of notation $\lambda:V \to V$ is the dilation $\lambda(x) = \lambda x$ for $\lambda > 0$.  Note that since $Q$ is fixed, the number $\lambda = \lambda(P)$ depends on the point $P$.  Following the same line of reasoning as in \cite{K2013}, we note that for $x, y, z, w, v_1, \ldots, v_k \in T_PM$ and for any $k$:
$$
\begin{array}{r c l}
\varphi(h\nabla^k\mathcal{R}_P(x, y, z; v_1, \ldots, v_k), hw) & = & \lambda^2 \varphi(F\nabla^k\mathcal{R}_P(x, y, z; v_1, \ldots, v_k), Fw) \\
 & = & \lambda^2g_P(\nabla^k\mathcal{R}_P(x, y, z; v_1, \ldots, v_k), w) \\
  & = & \lambda^2\nabla^kR_P(x, y, z, w; v_1, \ldots, v_k).
\end{array}
$$
On the other hand, we arrive at the following:
$$
\begin{array}{r c l}
\varphi(h\nabla^k\mathcal{R}_P(x, y, z; v_1, \ldots, v_k), hw) & = & \varphi (\mathcal{A}_k(hx, hy, hz; hv_1, \ldots, hv_k), hw) \\
 & = & A_k(hx, hy, hz, hw; hv_1, \ldots, hv_k) \\
 & = & \lambda^{k+4} A_k(Fx, Fy, Fz, Fw; Fv_1, \ldots, Fv_k).
\end{array}
$$
So, for every $k = 0, \ldots, r$, we have
$$
\nabla^kR_P = \lambda^{k+2}F^*A_k.
$$
Define $\psi(P) = \lambda(P)^2,$ so that $\lambda(P)^{k+2} = \psi(P)^{\frac{1}{2}(k+2)}$.  It is now clear the isometry $F$ is the needed isomorphism of model spaces, and the smoothness of $\psi$ is implied by the smoothness of $\nabla^kR$.

Conversely, suppose that there is a positive smooth function $\psi$ on $M$ and for every point $P \in M$, there is an isomorphism $F_P:T_PM \to V$ so that 
$$
\mathcal{M}_r(P) \cong (V, \varphi, \psi(P)A_0, \psi(P)^{3/2}A_1, \ldots, \psi(P)^{\frac{1}{2}(r+2)} A_r).
$$
It follows that $F_P^*A_k = \psi(P)^{-\frac{1}{2}(k+2)}\nabla^kR_P$, and that $(F_P^{-1})^*\nabla^kR_P = \psi(P)^{\frac12(k+2)}A_k$.

Now let $P, Q \in M$ be arbitrary, and let $F_P$ and $F_Q$ be the associated isomorphisms of model spaces.  The map $F = F_Q^{-1}F_P:T_PM \to T_QM$ is an isometry and satisfies
$$
\frac{\psi(P)^{\frac12(k+2)}}{\psi(Q)^{\frac12(k+2)}}F^*\nabla^kR_Q =  \nabla^kR_P.
$$
Define $\lambda = \sqrt{\frac{\psi(P)}{\psi(Q)}}$ and notice that $\lambda^{k+2}F^*\nabla^kR_Q = \nabla^k R_P$.  We claim the homothety $h = \lambda F$ satisfies $h^*\nabla^k \mathcal{R}_Q = \nabla^k\mathcal{R}_P$.  We follow a similar reasoning as in \cite{K2013}, consider arbitrary $x, y, z, w, v_1, \ldots, v_k \in T_PM$, and compute:
$$
\begin{array}{r c l}
g_Q(h(\nabla^k\mathcal{R}_P(x, y, z; v_1, \ldots, v_k)), hw) & = & \lambda^2 g_Q(F(\nabla^k\mathcal{R}_P(x, y, z; v_1, \ldots, v_k)), Fw) \\
 & = & \lambda^2 g_P(\nabla^k\mathcal{R}_P(x, y, z; v_1, \ldots, v_k), w) \\
 &=& \lambda^2 \nabla^kR_P(x, y, z, w; v_1, \ldots, v_k) \\
 &=& \lambda^{4+k}\nabla^kR_Q(Fx, Fy, Fz, Fw; Fv_1, \ldots Fv_k) \\
 & = & \nabla^kR_Q(hx, hy, hz, hw; hv_1, \ldots, hv_k) \\
 & = & g_Q(\nabla^k\mathcal{R}_Q(hx, hy, hz; hv_1, \ldots, hv_k), hw).
\end{array}
$$
Since the metric $g$ is nondegenerate, the above sequence of equalities demonstrate
$$
h(\nabla^k\mathcal{R}_P(x, y, z; v_1, \ldots, v_k)) = \nabla^k\mathcal{R}_Q(hx, hy, hz; hv_1, \ldots, hv_k),
$$
completing the proof that $h^*\nabla^k\mathcal{R}_Q = \nabla^k\mathcal{R}_P$ for every $k$.
\end{proof}

\section{The manifolds $(M, g_f)$}  \label{section3}

We begin our study of the manifolds $(M, g_f)$ by computing the covariant derivatives of the coordinate frames and the curvature.

\begin{lemma}  \label{lemmaMf}
Let $(M, g_f)$ and $\Delta$ be as in Section 1.
\begin{enumerate}
\item  The nonzero covariant derivatives of the coordinate frames are
$$
\nabla_{\partial_t}\partial_t = -f'e^{2f} \partial_y, \quad \nabla_{\partial_x} \partial_t = \nabla_{\partial_t}\partial_x = f'\partial_t.
$$
\item  The only potential nonzero curvature entries up to the usual symmetries are, for all $k$:
$$
\nabla^k R(\partial_x, \partial_t, \partial_t, \partial_x;\partial_x, \ldots, \partial_x) = -e^{2f}\Delta^{(k)}.
$$
\end{enumerate}
\end{lemma}

\begin{proof}
To prove Assertion (1), we note that the only possible Christoffel symbols of the second kind (up to the usual symmetries on coordinate frames) are 
$$
\Gamma_{ttx} = -f'e^{2f}, \quad {\rm and\ } \Gamma_{txt} = f'e^{2f}.
$$
Assertion (1) now follows.

To prove Assertion (2), we first compute the entries of $R$ on the coordinate frames.  Notice that if any entry of $R$ is $\partial_y$, then one gets $0$ as a result.  Thus up to symmetries we compute the only possible nonzero quantity as
$$
\begin{array}{r c l}
\mathcal{R}(\partial_x, \partial_t)\partial_t & = & \nabla_{\partial_x}\nabla_{\partial_t} \partial_t - \nabla_{\partial_t}\nabla_{\partial_x} \partial_t \\
& = & \nabla_{\partial_x}(-f'e^{2f}\partial_y )- \nabla_{\partial_t}(f'\partial_t) \\
 & = & -e^{2f}\Delta\partial_y.  {\rm\ So} \\
 R(\partial_x, \partial_t, \partial_t, \partial_x) & = & -e^{2f}\Delta.
\end{array}
$$
We prove the rest of Assertion (2) by induction, having already proven the base case.  Suppose that the only nonzero entry of $\nabla^kR$ is as given above.  There are two ways to produce possible nonzero curvature entries in $\nabla^{k+1}R$: either differentiate the components themselves, or produce a situation where the correct tangent vectors appear in the inputs of $\nabla^kR$.  The first strategy gives us only
$$\begin{array}{r c l}
\nabla^{k+1}R(\partial_x, \partial_t, \partial_t, \partial_x;\partial_x, \ldots, \partial_x) &=& \partial_x[-e^{2f}\Delta^{(k)}] - 2f'(-e^{2f}\Delta^{(k)}) \\
 & = & -e^{2f}\Delta^{(k+1)}.
 \end{array}
$$
The second strategy gives no possible nonzero entries since $\nabla_{\partial_t}\partial_x = f'\partial_t$ and this would produce a pairing of $\partial_t$ with $\partial_t$ in the first two or last two inputs of $\nabla^kR$.  Assertion (2) now follows.
\end{proof}

We may now prove Theorem \ref{Mf}.  

\begin{proof}[Proof of Assertions (1) and (2) of Theorem \ref{Mf}]
Suppose $\Delta > 0$, and define a new basis for $T_PM$ as follows, for some positive number $\lambda$:
\begin{equation}  \label{basisf}
T = e^{-f}\partial_t, \quad X = \lambda \partial_x, \quad Y = \frac{1}{\lambda}\partial_y.
\end{equation}
For every $\lambda$, the only nonzero metric quantities are constant and are $g_P(T,T) = g_P(X, Y) = 1$.  On this basis, we have
$$
R(T, X, X, T) = -\lambda^2\Delta.
$$
Setting $\lambda = |\Delta|^{-1/2}$ demonstrates that there is an isomorphism
$$
\mathcal{M}_0(P) \cong (V, \varphi, A),
$$
where $V = {\rm span}\{T, X, Y\}$, $\varphi(T, T) = \varphi(X, Y) = 1$, and $A(T, X, X, T) = -\frac{\Delta}{|\Delta|} = \pm 1$ are the only nonzero quantities up to the usual symmetries.  Assertion (1) follows.

To prove Assertion (2), consider setting $\lambda = 1$ in Equation (\ref{basisf}).  We again have the same metric relations (i.e., constant entries), and the only possible nonzero quantities of $\nabla^kR$ on this basis are
$$
\nabla^kR(T, X, X, T; X, \ldots, X) = -\Delta^{(k)}.
$$

Thus, if $\Delta^{(k)}$ are nonvanishing for all $k$, and if we define the tensor $A_k \in \otimes^{4+k} V^*$ on this basis to satisfy the following as its nonzero entries
$$
\begin{array}{r c l}
A_k(T, X, X, T; X, \ldots, X) &=& - A_k(T, X, T, X;X, \ldots, X) \\
& = & - A_k(X, T, X, T;X, \ldots, X)\\
 & = & A_k(X, T, T, X;X, \ldots, X) \\
  & = & -\frac{\Delta^{(k)}}{|\Delta^{(k)}|}
\end{array}
$$
and define $\psi_k = \Delta^{(k)}$, it follows that 
$$
(T_PM, g_P, \nabla^kR_P) \cong (V, \varphi, \psi_kA_k),
$$
and so $(M, g_f)$ is $CH_r(1,3)$ for all $r$.
\end{proof}

\begin{remark} {\rm 
It is interesting to note that in the above proof, the same isomorphism of model spaces is used to demonstrate
$$
(T_PM, g_P, \nabla^kR_P) \cong (V, \varphi, \psi_kA_k),
$$
although it does not follow that $(M,g_f)$ is $SCH_r(1,3)$ for all $r$, since it may not be the case that $\psi_k = \psi_0^{\frac12(k+2)}$ as in Theorem \ref{SCHchar}.  This possibility of being $SCH_r(1,3)$ of any order for this family is discussed in the remaining assertions of Theorem \ref{Mf}.  }
\end{remark}

In the proofs of Assertions (3) and (4) of Theorems \ref{Mf} and \ref{Mh}, we create various invariants of the model spaces in question.  Although it is not known if \emph{all} Weyl scalar invariants vanish,  it is easy to check that the scalar curvature, $||R||^2$ and $||\nabla R||^2$ vanish for this family, and so a more careful approach is required.  Our methods to construct invariants of model spaces will be of the same flavor as previous work (see \cite{D2009, D2004, G2007, GS2004} for some representative examples).  By definition (see page 26 of \cite{G2007}), if $\mathcal{M}$ is a model space, and $\Xi$ is a quantity depending on $\mathcal{M}$, then we say $\Xi$ is an \emph{invariant} of $\mathcal{M}$ if $\Xi(\mathcal{M}) = \Xi(\mathcal{M}')$ whenever $\mathcal{M} \cong \mathcal{M}'$.  We begin with a useful Lemma.

\begin{lemma}  \label{modeliso}
Let $\varepsilon \neq 0$.  Suppose $V = {\rm span}\{T, X, Y\}$, and an inner product $\varphi$ and algebraic curvature tensor $A$ have, up to the usual symmetries, nonzero entries given by
\begin{equation} \label{modelrelations1}
\varphi(T, T) = \varphi(X, Y) = 1, \quad A(T, X, X, T) = \varepsilon.
\end{equation}
If $F$ is an isomorphism of the model space $(V, \varphi, A)$, then there exist real numbers $a_1, ..., a_6$ where $a_1^2 = a_4^2 = 1$, and 
$$
\begin{array}{r c l}
FT &=& a_1 T   \hfill +  a_2 Y, \\
FX &=& a_3T + a_4X + a_5Y,  \\
FY &= &  \hfill a_6 Y.
\end{array}
$$
\end{lemma}
\begin{proof}
Any isomorphism $F$ of $(V, \varphi, A)$ must necessarily preserve the relations in Equation (\ref{modelrelations1}).  Since $F^*A = A$,  we have $AY = a_6 Y$ for some nonzero $a_6$.  Because $\varphi(FT, FY) = 0$, there must be nonzero real numbers $a_1$ and $a_2$ so that $FT = a_1T + a_2Y$.  Since $\varphi(T, T) = 1$, we must have $a_1^2 = 1$.  Now since $A(FT, FX, FX, FT) = \varepsilon \neq 0$, there must be real numbers $a_3, a_4,$ and $a_5$, with $a_4^2 = 1$ so that $FX = a_3T + a_4X + a_5Y$.
\end{proof}

\begin{proof}[Proofs of Assertions (3) and (4) of Theorem \ref{Mf}]
There is a basis $\{T, X, Y\}$ for $T_PM$, so that the only nonzero entries up to the usual symmetries of the tensors involved are
$$
\varphi(T, T) = \varphi(X, Y) = 1, \quad R_P(T, X, X, T) = -1.
$$
We use the basis given in Equation (\ref{basisf}), where $\lambda = |\Delta|^{-1/2}$.

We demonstrate presently that $\Xi = \nabla R(T, X, X, T;X)^2 = (\Delta^{(1)})^2$ is an invariant of the model space $\mathcal{M}_1(P)$.   On the basis given in Equation (\ref{basisf}), where $\lambda = |\Delta|^{-1/2}$, we have
$$
\nabla R(T, X, X, T; T) = 0, \quad \nabla R(T, X, X, T; X)^2 = \Xi.
$$
Notice that by Lemma \ref{modeliso}, $\Xi$ is preserved by any isomorphism of $\mathcal{M}_0(P)$, and so it is necessarily preserved by any isomorphism of $\mathcal{M}_1(P)$.  Hence, $\Xi$ is an invariant of $\mathcal{M}_1(P)$.

If  $(M, g_f)$ were $CH_1$, then  $\Xi$  must be constant.  By assumption, $\Xi$ is not, so $(M, g_f)$ is not $CH_1$, and is therefore not locally homogeneous.  This completes the proof of Assertion (3).

We now prove Assertion (4).  Suppose $f$ is analytic, and $(M, g_f)$ is $SCH_1$.  Therefore, according to Theorem \ref{SCHchar}, there exists a smooth positive function $\psi$ and isomorphisms
$$
\mathcal{M}_1(P) \cong (V, \varphi, \psi(P) A_0, \psi(P)^{3/2}A_1).
$$
Since $(M, g_f)$ is $CH_0(1,3)$, it must be the case that $A_0$ is the same curvature tensor given in Lemma \ref{modeliso}.  Therefore, there exists a basis $\{T, X, Y\}$ so that the only nonzero entries of these tensors up to the usual symmetries are given below:
$$
g_P(T, T) = g_P(X, Y) = 1, \quad R_P(T, X, X, T) = \pm\psi(P).
$$
We choose this basis to be $T = e^{-f}\partial_t, X = \partial_x, Y = \partial_y$ as in Equation (\ref{basisf}), so that $\psi = R_P(T, X, X, T) =  \pm\Delta,$ depending on the sign of $\Delta$.  On this basis, the nonzero entries of $\nabla R$ are again
$$
\nabla R(T, X, X, T; T) = 0, \quad \nabla R(T, X, X, T; X) = -\Delta^{(1)}.
$$

By assumption, $(M, g_f)$ is $SCH_1(1,3)$, and so there exists some basis $\{\bar T, \bar X, \bar Y\}$ for $T_PM$ so that  the nonzero entries of these tensors up to the usual symmetries are
\begin{equation}  \label{barrel}
\begin{array} {r c l}  
g_P(\bar T, \bar T) &=& g_P(\bar X, \bar Y) = 1, \\
R_P(\bar T, \bar X, \bar X, \bar T) & = & \pm \psi, \\
\nabla R_P(\bar T, \bar X, \bar X, \bar T; \bar X) & = & c_1 \psi^{3/2} {\rm\ or\ } 0, \\
\nabla R_P(\bar T, \bar X, \bar X, \bar T; \bar T) & = & c_2\psi^{3/2} {\rm\ or\ } 0,
\end{array}
\end{equation}
where $c_1,$ and $c_2$ are some nonzero constants.  Equation (\ref{barrel}) lists all possible nonzero entries of $R_P$ and $\nabla R_P$: the group of isomorphisms of $\mathcal{M}_0(P) \cong (V, \varphi, \psi A_0)$ is isomorphic to the group of isomorphisms of $(V, \varphi, A)$, where this model space is the one given in Lemma \ref{modeliso}.  The change of basis $\{T, X, Y\} \to \{\bar T, \bar X, \bar Y\}$ is necessarily an isomorphism of $\mathcal{M}_0(P)$, so we need not consider entries of $\nabla R_P$ that have a $\bar Y$ in any of its slots, as this would result in a zero entry.    

According to Lemma \ref{modeliso}, 
$$
\nabla R_P(\bar T, \bar X, \bar X, \bar T; \bar T) = 0, {\rm\ and\ } (\nabla R_P(\bar T, \bar X, \bar X, \bar T; \bar X))^2 = c_1^2\psi^3 = (\Delta^{(1)})^2.
$$
According to Lemma \ref{modeliso}, the following is an invariant of $\mathcal{M}_1(P)$:
$$
\Xi = \frac{\nabla R_P(\bar T, \bar X, \bar X, \bar T; \bar X)^2}{R_P(\bar T, \bar X, \bar X, \bar T)^3}.
$$
The assumption that $(M, g_f)$ is $SCH_r(1,3)$ and the relations in Equation (\ref{barrel}) force 
$$
\Xi = \frac{[c_1\psi^{3/2}]^2}{[\pm\psi]^3} = \pm c_1^2 {\rm\ to\ be\ constant.}
$$
But it has been established that 
$$\begin{array}{r c l}
R_P(\bar T, \bar X, \bar X, \bar T) &=& R_P(T, X, X, T) = - \Delta, {\rm\ and }\\
(\nabla R_P(\bar T, \bar X, \bar X, \bar T; \bar X))^2 &=&  (\Delta^{(1)})^2.
\end{array}$$
So $\Xi  = \frac{(\Delta^{(1)})^2}{-\Delta^3}$, or $\Delta^{3/2} = c\Delta^{(1)}$ for some nonzero constant $c$.  Inductively, one sees that there exists nonzero constants $\bar c_k$ so that
$$
\Delta^{(k)} = \bar c_k (\Delta)^{\frac{1}{2}(k+2)}.
$$
Additionally, setting $\tilde T = e^{-f}\partial_t, \tilde X = |\Delta|^{-1/2}\partial_x,$ and $\partial_y = |\Delta|^{1/2}$, we can, with assistance from Lemma \ref{lemmaMf}, list up to the usual symmetries the nonzero entries of the metric and all covariant derivatives of $R_P$:
$$
\begin{array}{r c l}
1 & = & g_P(\tilde T, \tilde T) =  g_P(\tilde X, \tilde Y) , {\rm\ and } \\
\bar c_k & = & \nabla^k R_P(\tilde T, \tilde X, \tilde X, \tilde T; \tilde X, \ldots, \tilde X).
\end{array}
$$
Therefore $(M, g_f)$ is $CH_r$ for all $r$.  In the analytic category, this implies $(M, g_f)$ is locally homogeneous (see page 91 of \cite{G2007}).
\end{proof}

\section{The manifolds $(M,g_h)$}  \label{section4}

As in the last section, we present a lemma which exhibits the relevant curvature information.

\begin{lemma}  \label{lemmaMh}
Let $(M, g_h)$ be as in Section 1.
\begin{enumerate}
\item  The nonzero covariant derivatives of the coordinate frames are
$$
\nabla_{\partial_x}\partial_x = h'\partial_t,\quad \nabla_{\partial_x}\partial_t = -h'\partial_y.
$$
\item  The only potential nonzero curvature entries of $R, \nabla R$, and $\nabla^2 R$ up to the usual symmetries:
$$
\begin{array}{r c l}
R(\partial_t, \partial_x, \partial_x, \partial_t) & = & h'', \\
\nabla R(\partial_t, \partial_x, \partial_x, \partial_t; \partial_t) &=& h''', \\
\nabla^2 R(\partial_t, \partial_x, \partial_x, \partial_t; \partial_t, \partial_t) &=& h^{(4)}, \\
\nabla^2 R(\partial_t, \partial_x, \partial_x, \partial_t; \partial_x, \partial_x) &=& h' h^{(3)}.
\end{array}
$$
\end{enumerate}
\end{lemma}
\begin{proof}
These computations are very similar in nature to the calculations done for Lemma \ref{lemmaMf}, and are omitted.
\end{proof}

\begin{proof}[Proof of Assertions (1) and (2) of Theorem \ref{Mh}]
Let $P \in M$ be given.  We define a new basis $\{T, X, Y\}$ for $T_PM$ as
\begin{equation}  \label{basish}
T = \partial_t, \quad X = \partial_x + h\partial_y, \quad Y = \partial_y.
\end{equation}
This basis has the constant metric entries
$$
g_P(T, T) = g_P(X, Y) = 1.
$$
Moreover, by setting 
\begin{equation}  \label{basish2}
\bar T = T, \bar X = \lambda X, {\rm\ and\ } \bar Y = \frac{1}{\lambda} Y,
\end{equation} 
such a change preserves these constant metric entries, and the potentially nonzero entries of $R$ and $\nabla R$ up to the usual symmetries are
\begin{equation}  \label{basisexpression}
R_P(\bar T, \bar X, \bar X, \bar T) = \lambda^2 h'', \quad \nabla R(\bar T, \bar X, \bar X, \bar T; \bar T) = \lambda^2 h'''.
\end{equation}
We may now prove Assertion (1).  Suppose $h''$ is nonvanishing.  Setting $\lambda = \frac{1}{\sqrt{|h''|}}$ demonstrates that $(M, g_h)$ is $CH_0$.

To prove Assertion (2), in observance of Equation (\ref{basisexpression}), we simply look for a $\lambda$ and positive smooth function $\psi$ so that 
\begin{equation}  \label{sch1}
\lambda^2 h'' = \psi,\quad {\rm\ and\ } \quad \lambda^2 h''' = \psi^{3/2}.
\end{equation}
Assume that $h''$ and $h'''$ are nonvanishing.  Equation (\ref{sch1}) can be solved by setting 
\begin{equation}  \label{lambdapsi}
\lambda^2  = \frac{(h''')^2}{(h'')^3},  {\rm\ so\ that\ } \psi = \frac{(h''')^2}{(h'')^2} > 0.
\end{equation}
According to Theorem \ref{SCHchar}, we have demonstrated that $(M, g_h)$ is $SCH_1(1,3)$, completing the proof of Assertion (2).
\end{proof}

Similarly to the proof of Assertions (3) and (4) of Theorem \ref{Mf}, we shall employ the use of invariants to establish Assertions (3) and (4) of Theorem \ref{Mh}.  We first prove Assertion (3) of Theorem \ref{Mh}, and will prove a helpful Lemma before proving Assertion (4).

\begin{proof}[Proof of Assertion (3) of Theorem \ref{Mh}]
Let $P \in M$ be given.  In Equation (\ref{basish}), we constructed the basis $\{\bar T, \bar X, \bar Y\}$ for which the metric and curvature tensor has the following potential nonzero entries for $\lambda = \frac{1}{\sqrt{|h''|}}$:
$$
g_P(\bar T, \bar T) = g_P(\bar X, \bar X) = (R_P(\bar T, \bar X, \bar X, \bar T))^2 = 1.
$$
Thus, the model space $\mathcal{M}_0(P)$ is of the form considered in Lemma \ref{modeliso} where $\varepsilon$ is either $+1$ or $-1$.  Consider then the model space $\mathcal{M}_1(P)$
and the quantity 
$$
\Xi = \nabla R_P (\bar T, \bar X, \bar X, \bar T; \bar T)^2 = \frac{1}{(h'')^2} \cdot (h''')^3.
$$
According to Lemma \ref{modeliso}, $\Xi$ is invariant under isomorphisms of $\mathcal{M}_0(P)$, and is therefore invariant under isomorphisms of $\mathcal{M}_1(P)$, and is therefore an invariant of the model space $\mathcal{M}_1(P)$.  If $(M, g_h)$ is $CH_1$, then $\Xi$ must be constant.  This establishes Assertion (3).
\end{proof}

To prove Assertion (4) of Theorem \ref{Mh}, we will need a lemma which describes the group of isomorphisms of a particular model space.

\begin{lemma}  \label{1modeliso}
Let $\varepsilon_0, \varepsilon_1 \neq 0$.  Suppose $V = {\rm span}\{T, X, Y\}$, and an inner product $\varphi$ and algebraic curvature tensor $A$ have, up to the usual symmetries, nonzero entries given by
\begin{equation} \label{modelrelations}
\varphi(T, T) = \varphi(X, Y) = 1, \quad A(T, X, X, T) = \varepsilon_0.
\end{equation}
Suppose also that $A_1$ is an algebraic covariant derivative curvature tensor and has the following nonzero entry, up to the usual symmetries:
$$
A_1(T, X, X, T; T) = \varepsilon_1.
$$
If $F$ is an isomorphism of the model space $(V, \varphi, A, A_1)$, then there exist real numbers $b_1, ..., b_4$ where $b_2^2 = 1$, and 
$$
\begin{array}{r c l}
FT &=& T   \phantom{+b_i X} +  b_1 Y, \\
FX &=&  \hfill b_2X + b_3Y,  \\
FY &= &  \hfill b_4 Y.
\end{array}
$$
\end{lemma}
\begin{proof}
Since any isomorphism of this model space is necessarily an isomorphism of $(V, \varphi, A)$, we may apply Lemma \ref{modeliso} to conclude all of what is needed, except possibly that there exists a constant $b$ with $FX = bT + b_3 X + b_4 Y$, or that the coefficient $a$ of $T$ in the expression for $FT$ is equal to 1 (as opposed to squaring to 1 as in Lemma \ref{modeliso}).  But $F^* A_1 = A_1$, and so
$$
0 = A_1(FT, FX, FX, FT; FX) = b\varepsilon_1,
$$
so $b = 0$.  In addition,
$$
\varepsilon_1 = A_1(FT, FX, FX, FT; FT) = a\varepsilon_1,
$$
so $a = 1$, completing the proof.
\end{proof}

We can now establish Assertion (4) of Theorem \ref{Mh}.

\begin{proof}[Proof of Assertion (4) of Theorem \ref{Mh}]
Suppose $(M, g_h)$ is $SCH_2(1,3)$, and let $P \in M$.  There exists a positive smooth function $\psi$ and an isomorphism
$$
\mathcal{M}_2(P) \cong (V, \varphi, \psi(P)A_0, \psi(P)^{3/2}A_1, \psi(P)^2 A_2).
$$
As in the proof of Assertion (2) of Theorem \ref{Mh}, we let $\{\bar T, \bar X, \bar Y\}$ be the basis for $T_PM$ from Equation (\ref{basish2}), where $\lambda$ is defined as in Equation (\ref{lambdapsi}),  so that the following are the only possible nonzero tensor entries, up to the usual symmetries:
$$
\begin{array}{r c l}
1 &=& g_P(\bar T, \bar T) = g_P(\bar X, \bar Y), \\
\psi & = & R(\bar T, \bar X, \bar X, \bar T), \\
\psi^{3/2} & = & \nabla R(\bar T, \bar X, \bar X, \bar T; \bar T).
\end{array}
$$
Assume $h'''$ is nonvanishing.  Thus, the tensors $A_0$ and $A_1$ are the same as those given in Lemma \ref{1modeliso}, and so the assumption that $(M, g_h)$ is $SCH_2(1,3)$ gives us a basis $\{\tilde T, \tilde X, \tilde Y\}$ so that 

\begin{equation}  \label{rel1}
\begin{array}{r c l}
1 &=& g_P(\tilde T, \tilde T) = g_P(\tilde X, \tilde Y),\\
\psi & = & R(\tilde T, \tilde X, \tilde X, \tilde T), \\
\psi^{3/2} & = & \nabla R(\tilde T, \tilde X, \tilde X, \tilde T; \tilde T).
\end{array}
\end{equation}
Such a change of basis $\{\bar T, \bar X, \bar Y\} \to \{ \tilde T, \tilde X, \tilde Y\}$ is necessarily an isomorphism of $\mathcal{M}_1(P)$, and so using Lemma \ref{lemmaMh} and Lemma \ref{1modeliso}, the following are the only possible nonzero entries of $\nabla^2 R$ on such a basis:
\begin{equation}  \label{rel2}
\begin{array}{r c l}
\Xi_T := \nabla^2 R(\tilde T, \tilde X, \tilde X, \tilde T; \tilde T, \tilde T) & = &c_1 \psi^2 {\rm\ or\ } 0, \\
\Xi_X:= \nabla^2 R(\tilde T, \tilde X, \tilde X, \tilde T; \tilde X, \tilde X) & = &c_2 \psi^2 {\rm\ or\ } 0,
\end{array}
\end{equation}
for some nonzero constants $c_1, c_2$.

The quantities
$$
\xi_T:= \frac{\Xi_T}{R(\tilde T, \tilde X, \tilde X, \tilde T)^2} \quad {\rm\ and} \quad \xi_X:= \frac{\Xi_X}{R(\tilde T, \tilde X, \tilde X, \tilde T)^2}
$$
are therefore both invariants of $\mathcal{M}_2(P)$ as both are invariant under any isomorphism of the model $\mathcal{M}_2(P)$ by Lemma \ref{1modeliso}, since any such isomorphism is also an isomorphism of $\mathcal{M}_1(P)$.  Moreover, both $\xi_T$ and $\xi_X$ are constants, considering the relations in Equations (\ref{rel1}) and (\ref{rel2}) on this basis, listed above.

Evaluating $\xi_T$ and $\xi_X$ on the basis $\{\tilde T, \tilde X, \tilde Y\}$ gives us
$$
\xi_T = \frac{h^{(4)}}{(h'')^2}, \quad \xi_X = \frac{h''' h' }{(h'')^2} \neq 0.
$$
We remark that $\xi_X$ is nonzero by assumption, and integrating the differential equation $\xi_X\frac{h''}{h'} = \frac{h'''}{h''}$ for nonzero constant $\xi_X$ reveals that $h''$ is proportional to $h'$.  Differentiating this proportionality, we conclude that there exists nonzero constant $c$ with $h''' = ch'',$ and $h^{(4)} = c^2h''$.  For this particular $h$, we summarize the nonzero curvature and metric information on the basis $\{T, X, Y\}$, where $T = \partial_t$, $X = \frac{1}{\sqrt{|h''|}}(\partial_x + h \partial_y)$, and $Y = \sqrt{|h''|}\partial_y$:

$$
\begin{array}{r c l}
1 & = & g_P(T, T) = g_P(X, Y), \\
\sigma & = & R(T, X, X, T) {\rm\ for\ } \sigma = \pm 1,\\
\sigma c & = & \nabla R(T, X, X, T; T), \\
\sigma c^2 & = & \nabla^2 R(T, X, X, T;T, T), \\
\xi_X & = & \nabla^2 R(T, X, X, T; X, X).
\end{array}
$$
In other words, $(M, g_h)$ is $CH_2$.  In the analytic category, this implies local homogeneity \cite{BD}.

\end{proof}

\section{Perspectives and open questions}

In Singer's original paper \cite{S1960} one sees that he arrives at Theorem \ref{Singer} by pointing out that as the level of curvature homogeneity rises, the (Lie) group of model space isomorphisms strictly decreases in dimension (see also the introduction of  \cite{BD}).  Since these groups of isomorphisms are nested by the level curvature homogeneity, this result makes sense.  In the definition of $CH_r(1,3)$, there is no requirement that a group of model space isomorphisms be nested in the same way and, as such, it is not surprising that there is no finite Singer-type invariant for this generalization of curvature homogeneity (see Assertions (2)  and (3) of Theorem \ref{Mf}).

On the other hand, the property of $SCH_r(1,3)$ \emph{does} have a nested group of model space isomorphisms, and is more strongly related to the property of being $CH_r$.  In fact, the isomorphisms of $$\mathcal{M}_r(P) \cong (V, \varphi, \psi A_0, \ldots, \psi^{\frac{1}{2}(r+2)}A_r)$$ 
are exactly that of $(V, \varphi, A_0, \ldots, A_r),$ which was used in this paper (see the proof of Assertion (4) of Theorems \ref{Mf} and \ref{Mh}).  It is for this reason that we conjecture the existence of a number $s_{p,q}$, depending on the signature of the manifold, so that if it is $SCH_{s_{p,q}}(1,3)$, then it is locally homogeneous.  In lieu of these observations, we conjecture that $s_{p,q} \leq k_{p,q}$, and our examples in this paper support this conjecture.

\section*{Acknowledgments}
This research was jointly funded by NSF grant DMS-1156608, and by California State University, San Bernardino.

\end{document}